\documentclass[]{article}
\usepackage{amsmath, amsfonts, fullpage, xspace, color, amsthm, hyperref}

\newtheorem{lemma}{Lemma}[section]

\newtheorem{theorem}[lemma]{Theorem}
\newtheorem{prop}[lemma]{Proposition}

\newtheorem{hyp}{Hypothesis}

\theoremstyle{definition}
\newtheorem{defn}[lemma]{Definition}
\newtheorem{example}[lemma]{Example}
\newtheorem*{example*}{Example}

\numberwithin{equation}{section}

\title{Characteristic matrix functions for delay differential equations with symmetry}

\date{}
\author{Babette de Wolff\footnote{Vrije Universiteit Amsterdam, Department of Mathematics, \href{mailto:b.wolff@vu.nl}{b.wolff@vu.nl}.}}

\begin{document}
\maketitle

\begin{abstract}
\noindent
A characteristic matrix function captures the spectral information of a bounded linear operator in a matrix-valued function. 
In this article, we consider a delay differential equation with one discrete time delay and assume this equation is equivariant with respect to a compact symmetry group.
Under this assumption, the delay differential equation can have discrete wave solutions, i.e. periodic solutions that have a discrete group of spatio-temporal symmetries. We show that if a discrete wave solution has a period that is rationally related to the time delay, then we can determine its stability using a characteristic matrix function. The proof relies on equivariant Floquet theory and results by Kaashoek and Verduyn Lunel on characteristic matrix functions for classes of compact operators.  
We discuss applications of our result in the context of delayed feedback stabilization of periodic orbits. 

\medskip

\end{abstract} 

\section{Introduction}
For infinite dimensional dynamical systems, determining the stability of an invariant set often poses challenges due to the infinite dimensional nature of the problem. However, sometimes we are able to make a dimension reduction in the sense that the stability of the invariant set can be determined by computing zeros of a scalar valued function. 
This simplifies the stability analysis since we can now apply analytical and numerical techniques directly to the scalar valued function. 
Concrete examples of this approach appear in the context of partial differential equations, where the stability of certain travelling wave solutions can be computed using the Evans function \cite{Sandstede}. For delay differential equations (DDE), the stability of equilibria and certain classes of periodic orbits can be computed using so-called characteristic matrix functions, as introduced by Kaashoek and Verduyn Lunel in \cite{KaashoekVL92} and \cite{KaashoekVL20}. 

This article is concerned with DDE that have built-in symmetries; 
in this case, periodic solutions of the DDE
can satisfy additional spatio-temporal relations. 
Although spatio-temporal patterns and their stability are well studied in the context of ordinary differential equations (see e.g. \cite{Fiedler88, Lamb99, Wulff2006}),
they are much less explored in the setting of DDE. 
This article makes a next step in the stability analysis of spatio-temporal patterns in DDE by combining the concept of characteristic matrix functions with techniques from symmetric systems. 

Specifically, we consider a DDE of the form 
\begin{equation}\label{eq:dde intro}
 \dot{x}(t) = f(x(t), x(t-\tau) ), \qquad t \geq 0
\end{equation}
with $f: \mathbb{R}^N \times \mathbb{R}^N \to \mathbb{R}^N$ a $C^2$ function and time delay $\tau > 0$. We assume that the DDE \eqref{eq:dde intro} is equivariant with respect to a compact subgroup $\Gamma \subseteq \mathrm{GL}(N, \mathbb{R})$ of the general linear group. 
In this case, a periodic orbit $x_\ast$ of \eqref{eq:dde intro} can satisfy spatio-temporal relations of the form $hx_\ast(t) = x_\ast(t+r)$, with $h \in \Gamma$ a spatial transformation and $r > 0$ a fraction of the period. We prove that if the time shift $r$ is equal to the time delay $\tau$, we can determine the stability of $x_\ast$ by computing the zeroes of a scalar valued function. 

In the situation where the DDE \eqref{eq:dde intro} is \emph{not} symmetric but does have a periodic orbit with period equal to the delay, we can determine the stability of this periodic orbit by computing the roots of a scalar valued function. 
This result was proven under an additional condition in \cite{VL92} and later in full generality 
in \cite[Section 11.4]{KaashoekVL20}. The result presented in this article can be viewed as a refinement of the result in \cite{VL92} and \cite[Section 11.4]{KaashoekVL20} in the sense that under the extra assumption that the DDE is symmetric, we are able to make more precise statements about periodic solutions with spatio-temporal patterns. 

The result presented in this article is particularly relevant in the context of \emph{equivariant Pyragas control}, a delayed feedback control scheme that aims to stabilize spatio-temporal patterns. 
If the ordinary differential equation
\begin{equation} \label{eq:ode intro}
\dot{x}(t) = F(x(t)), \qquad F: \mathbb{R}^N \to \mathbb{R}^N 
\end{equation}
has an unstable periodic solution $x_\ast$ with spatio-temporal relation $hx_\ast(t) = x_\ast(t+r)$, then this is also a solution of the delay differential equation
\begin{equation} \label{eq:control intro}
\dot{x}(t) = F(x(t)) + K \left[x(t) - hx(t-r) \right], \qquad K \in \mathbb{R}^{N \times N},
\end{equation}
cf. \cite{Fiedler10}. However, the overall dynamics of systems \eqref{eq:ode intro} and \eqref{eq:control intro} are radically different, and we can try to choose the matrix $K \in \mathbb{R}^{N \times N}$ in such a way that $x_\ast$ is a stable solution of \eqref{eq:control intro}. Despite its many experimental applications (see e.g. \cite{Schoell19, Schneider21}),
mathematical results on the control scheme \eqref{eq:control intro} are rare 
due to the periodic and infinite dimensional nature of the problem.
In fact, most analytical results in the literature so far 
are either close to a bifurcation point \cite{Hooton19, Hooton17, deWolff2017, Fiedler20, refuting} or concern periodic orbits that can be transformed to stationary solutions of autonomous systems \cite{Purewal14,Fiedler08,equivariant,Fiedler10,Fiedlerconstraint}. The contents of this article are a first step towards further insights in equivariant Pyragas control, such as the results on stabilization of non-stationary periodic orbits far away from bifurcation point presented by the author in \cite{proefschrift}. 

We start the rest of this article by formally stating its main result and introducing the necessary terminology in Section \ref{sec:main result}. Section \ref{sec:background cm} then reviews material from \cite{KaashoekVL92} and \cite{KaashoekVL20} on characterstic matrices. Section \ref{sec:proof} contains the proof of the main result. In Section \ref{sec:applications} we further discuss applications of the result to delayed feedback control of periodic orbits.

\subsection*{Acknowledgements}
The contents of this article are based upon contents of the authors doctoral thesis \cite{proefschrift}, written at the Freie Universit{\"a}t Berlin under the supervision of Bernold Fiedler. The author was partially supported by the Berlin Mathematical School and was an associated member of SFB 910 ``Control of self-organizing linear systems''.

The author is grateful to Bernold Fiedler and Sjoerd Verduyn Lunel for useful discussions and encouragement; and to Alejandro L{\'o}pez Nieto, Bob Rink and Isabelle Schneider for comments on earlier versions.

\section{Setting \& statement of the main result} \label{sec:main result}

Throughout this section, we consider the DDE \eqref{eq:dde intro} and assume that this DDE has a periodic orbit and is symmetric with respect to a compact subgroup of the general linear group. We summarize this in the following hypothesis: 

\begin{hyp} \hfill \label{hyp:symmetries dde}
\begin{enumerate}
\item DDE \eqref{eq:dde intro} has a periodic orbit $x_\ast$ with minimal period $p > 0$; 
\item DDE \eqref{eq:dde intro} is equivariant with respect to a compact subgroup $\Gamma$ of the general linear group $\mathrm{GL}(N, \mathbb{R})$, i.e.
\begin{equation} \label{eq:equivariance}
f(\gamma x, \gamma y) = \gamma f(x, y) \qquad \mbox{for all } x, y \in \mathbb{R}^N \mbox{ and } \gamma \in \Gamma. 
\end{equation}
\end{enumerate}
\end{hyp}

If $x(t)$ is a solution of the DDE \eqref{eq:dde intro} and $\gamma$ is an element of $\Gamma$, then the equivariance relation \eqref{eq:equivariance} implies that $\gamma x(t)$ is a solution of the DDE \eqref{eq:dde intro} as well. So we can view $\Gamma$ as a group of symmetries of the solutions of \eqref{eq:dde intro}. Moreover, the equivariance relation \eqref{eq:equivariance} naturally induces two symmetry groups on the periodic orbit $x_\ast$; we define the groups: 
\begin{subequations}
\begin{align}
K_\ast &= \{ \gamma \in \Gamma \mid \gamma x_\ast(t) = x_\ast(t) \ \mbox{for all} \ t \in \mathbb{R} \}; \\
H_\ast & = \{ \gamma \in \Gamma \mid \mbox{there exists a } \Theta(\gamma) \in [0, 1) \mbox{ such that } \gamma x_\ast(t) = x_\ast(t+\Theta(\gamma)p) \mbox{ for all }t \in \mathbb{R} \}. 
\end{align}
\end{subequations} 
The elements of group $K_\ast$ leave the periodic orbit fixed pointwise; therefore we refer to $K_\ast$ as the group of \textbf{spatial symmetries} of $x_\ast$. Since every element $h \in H_\ast$ induces a spatio-temporal relation of the form 
$ hx_\ast(t) = x_\ast(t+\Theta(h)p)$ 
on the periodic solution, we refer to $H_\ast$ as the group of \textbf{spatio-temporal symmetries} of $x_\ast$.

If $h_1, h_2 \in H_\ast$ are two spatio-temporal symmetries of $x_\ast$, then $h_1 h_2 x_\ast(t) = x_\ast(t+\Theta(h_1)p + \Theta(h_2)p)$
and hence $\Theta(h_1 h_2) = \Theta(h_1) + \Theta(h_2) \mod 1$. Thus the map
\[ \Theta: H_\ast \to S^1 \simeq \mathbb{R} / \mathbb{Z}. \]
is a group homomorphism and $K_\ast = \ker \Theta$ is a normal subgroup of $H_\ast$. Therefore $H_\ast/K_\ast \simeq \mbox{im } \Theta$
and $\mbox{im } \Theta$ is a subgroup of $S^1$. This implies that 
\begin{align*}
\begin{cases}
H_\ast/K_\ast \simeq \mathbb{Z}_m \qquad \mbox{for some } m \in \mathbb{N}, \mbox{ or } \\
H_\ast/K_\ast \simeq S^1,
\end{cases}
\end{align*} 
where $\mathbb{Z}_m$ denotes the cyclic group of order $m$. If $H_\ast/K_\ast \simeq S^1$, we say that the periodic solution $x_\ast$ is a \textbf{rotating wave}; if $H_\ast/K_\ast \simeq \mathbb{Z}_m$ (i.e. the group of spatio-temporal symmetries modulo the group of purely spatial ones is a finite group), we say that the periodic solution $x_\ast$ is a \textbf{discrete wave}; cf. \cite{Fiedler88}.  
 
\medskip

To determine whether $x_\ast$ is a stable solution of the DDE \eqref{eq:dde intro}, we consider the linearized system
\begin{subequations}
\begin{equation} \label{eq:linearized dde}
\dot{y}(t) = \partial_1 f(x_\ast(t), x_\ast(t-\tau)) y(t) + \partial_2 f(x_\ast(t), x_\ast(t-\tau)) y(t-\tau). 
\end{equation}
If we supplement \eqref{eq:linearized dde} with the initial condition
\begin{equation} \label{eq:initial condition}
y(s+t) = \varphi(t) \qquad \mbox{for } t \in [-\tau, 0] \mbox{ and } \phi \in C \left([-\tau, 0], \mathbb{R}^N\right)
\end{equation}
\end{subequations}
then the system \eqref{eq:linearized dde}--\eqref{eq:initial condition} has a unique solution $y(t)$ for $t \geq s$. We define the \textbf{history segment} $y_t \in C \left([-\tau, 0], \mathbb{R}^N\right)$ of this solution as
$y_t(\theta) = y(t+\theta)$. 
We then associate to \eqref{eq:linearized dde}--\eqref{eq:initial condition} a two-parameter system of operators 
\begin{equation} \label{eq:two parameter system}
U(t, s): C\left([-\tau, 0], \mathbb{R}^N \right) \to C\left([-\tau, 0], \mathbb{R}^N \right), \qquad t \geq s
\end{equation}
defined via the relation $y_t = U(t, s) \phi$. 
We refer to 
\eqref{eq:two parameter system}
as the \textbf{family of solution operators}
of \eqref{eq:linearized dde}, cf. \cite[Chapter 12]{delayequations}. 

Floquet theory for DDE implies that the non-zero spectrum of the \textbf{monodromy operator} $U(p, 0)$ consists of isolated eigenvalues of finite algebraic multiplicity; these eigenvalues of $U(p, 0)$ determine the stability of the periodic solution $x_\ast$ \cite[Chapter 13]{delayequations}. 
The equivariance assumption in Hypothesis \ref{hyp:symmetries dde} allows us to refine Floquet theory for discrete waves. We make this precise in the following proposition, which we cite without proof from \cite{proefschrift}. 
The statement of the proposition is analogous to the formulation of equivariant Floquet theory of ODE (cf. \cite{Wulff2006}), but the proof now also involves compactness of the relevant operator. 

Throughout, we let an element $\gamma \in \Gamma \subseteq \mathrm{GL}(N, \mathbb{R})$ act on the state space $C\left([-\tau, 0], \mathbb{R}^N\right)$ via $(\gamma \phi)(\theta) = \gamma \phi(\theta)$ for $\phi \in C \left([-\tau, 0], \mathbb{R}^N \right)$ and $\theta \in [-\tau, 0]$.

\begin{prop}[Stability of discrete waves,  {\cite[Proposition 6.3]{proefschrift}}] 
\label{prop:reduced mon op}
Consider the DDE \eqref{eq:dde intro} satisfying Hypothesis \ref{hyp:symmetries dde}, and additionally assume that the periodic solution $x_\ast$ is a discrete wave. Let $U(t,s), \ t \geq s$ be the family of solution operators of 
the linearized problem \eqref{eq:linearized dde}. For $h \in H_\ast$ a spatio-temporal symmetry of $x_\ast$, define the operator
\begin{equation} \label{eq:reduced mon op}
U_h = h^{-1} U(\Theta(h)p, 0).
\end{equation}
Then the following statements hold:
\begin{enumerate}
\item The non-zero spectrum of $U_h$ consists of isolated eigenvalues of finite algebraic multiplicity;
\item $1 \in \mathbb{C}$ is an eigenvalue of $U_h$; 
\item If $U_h$ has an eigenvalue strictly outside the unit circle, then $x_\ast$ is an unstable solution of \eqref{eq:dde intro}. If the eigenvalue $1 \in \sigma_{pt}(U_h)$ is algebraically simple and all other eigenvalues of $U_h$ lie strictly inside the unit circle, then $x_\ast$ is a stable solution of \eqref{eq:dde intro}. 
\end{enumerate}
\end{prop}

We aim to determine the eigenvalues of the operator $U_h$ in \eqref{eq:reduced mon op} 
using the concept of a \textbf{characteristic matrix function} for bounded linear operators, as introduced by Kaashoek and Verduyn Lunel in \cite{KaashoekVL20}.
We denote by $\mathcal{L}(X,X)$ the space of bounded linear operators on a complex Banach space $X$. Moreover, we denote by $I_X$ the identity operator on a Banach space $X$, but supress the subscript when the underlying space is clear. 

\begin{defn} [{\cite[Definition 5.2.1]{KaashoekVL20}}] \label{defn:cm} Let $X$ be a complex Banach space, $T: X \to X$ be a bounded linear operator and $\Delta: \mathbb{C} \to \mathbb{C}^{n \times n}$ be an analytic matrix-valued function. We say that $\Delta$ is a \textbf{characteristic matrix function} for $T$ if there exist analytic functions
\[ E, F: \mathbb{C} \to \mathcal{L}\left(\mathbb{C}^n \oplus X, \mathbb{C}^n \oplus X \right) \]
such that $E(z), F(z)$ are invertible operators for all $z \in \mathbb{C}$ and such that 
\begin{equation}\label{eq:defn cm}
\begin{pmatrix}
\Delta(z) & 0 \\
0 & I_{X}
\end{pmatrix} = F(z) \begin{pmatrix}
I_{\mathbb{C}^n} & 0 \\
0 & I - zT
\end{pmatrix} E(z)
\end{equation}
holds for all $z \in \mathbb{C}$. 
\end{defn}

The idea of the above definition is to make a conjugation between the analytic function
\[\mathbb{C} \ni z \mapsto I - zT \]
and the analytic function
\[ z \mapsto \Delta(z) \in \mathbb{C}^{n \times n}. \] 
However, this cannot be done directly, since in general the dimensions of $X$ and $\mathbb{C}^{n \times n}$ are not the same. So we first (trivially) extend the functions $z \mapsto I-zT$ and $z \mapsto \Delta(z)$ to the functions
\begin{equation} \label{eq:extension}
\qquad z \mapsto \begin{pmatrix}
I_{\mathbb{C}^n} & 0 \\
0 & I - zT
\end{pmatrix} \in \mathcal{L}\left(\mathbb{C}^n \oplus X,\mathbb{C}^n \oplus X \right), \qquad 
z \mapsto \begin{pmatrix}
\Delta(z) & 0 \\
0 & I_{X}
\end{pmatrix} \in \mathcal{L} \left(\mathbb{C}^n \oplus X, \mathbb{C}^n \oplus X\right)
\end{equation}
respectively. If now the functions in \eqref{eq:extension} are related via multiplication by analytic functions whose values are invertible operators, then this directly relates kernelvectors of $I-zT$ to kernelvectors of $\Delta(z)$. In particular, 
zeroes of the scalar valued function $z \mapsto \det \Delta(z)$ give information on the non-zero spectrum of $T$. 
We make this precise in Section \ref{sec:background cm}. We first state the main result of this article, which gives an explicit characteristic matrix for the operator $U_h$ defined in \eqref{eq:reduced mon op} in case the time shift $\Theta(h)p$ of the spatio-temporal pattern is equal to the time delay $\tau$.

\begin{theorem}[Main result]
\label{thm:main result}
Consider the DDE 
\begin{equation} \label{eq:dde thm 2}
\dot{x}(t) = f(x(t), x(t-\tau))
\end{equation}
with $f: \mathbb{R}^N \times \mathbb{R}^N \to \mathbb{R}^N$ a $C^2$-function and with time delay $\tau > 0$. Assume that
\begin{enumerate}
\item system \eqref{eq:dde thm 2} has a periodic solution $x_\ast$ with minimal period $p > 0$;
\item system \eqref{eq:dde thm 2} is equivariant with respect to a compact subgroup $\Gamma$ of the general linear group $\mathrm{GL}(N, \mathbb{R})$;
\item the periodic solution $x_\ast$ is a discrete wave and there exists a spatio-temporal symmetry $h \in H_\ast$ with
\[ hx_\ast(t) = x_\ast(t+\tau). \]
\end{enumerate}
Let $U(t, s), \ t \geq s$ be the family of solution operators of the linearized DDE 
\begin{equation} \label{eq:linearized dde thm}
\dot{y}(t) = \partial_1 f(x_\ast(t), x_\ast(t-\tau))y(t) +  \partial_2 f(x_\ast(t), x_\ast(t-\tau)) y(t-\tau). 
\end{equation}
For $z \in \mathbb{C}$, let $F(t, z)$ be the fundamental solution of the ODE
\[ \dot{y}(t) = \partial_1 f(x_\ast(t), x_\ast(t-\tau))y(t) +  z \cdot \partial_2 f(x_\ast(t), x_\ast(t-\tau)) y(t) \]
with $F(0, z) = I_{\mathbb{C}^N}$. 
Then the analytic function
\[ \Delta(z) = I_{\mathbb{C}^N} - zh^{-1} F(\Theta(h)p, z) \]
is a characteristic matrix function for the operator
\[ U_h = h^{-1} U(\Theta(h)p, 0). \]
\end{theorem}

\section{Characteristic matrices \&  Spectral information} \label{sec:background cm}
This section reviews material from \cite[Section 1]{KaashoekVL92} and \cite[Chapter 5]{KaashoekVL20} on characteristic matrix functions. The first part of this section (page \pageref{defn:Jordan chain}--\pageref{lem:eigenvalues cm}) discusses how characteristic matrix functions capture the spectrum of a bounded linear operator. We have included the contents of page \pageref{defn:Jordan chain}--\pageref{lem:eigenvalues cm} in this article to give context to Theorem \ref{thm:main result} and to illustrate the implications of this theorem; we discuss its applications further in Section \ref{sec:applications}. 
In the second part of this section (page \pageref{thm: CM V + R}), we state a theorem from \cite{KaashoekVL20} that constructs a characteristic matrix function for a class of compact operators. This theorem is the cornerstone for the proof of Theorem \ref{thm:main result} and is therefore crucial for the rest of this article. 

We start by recalling the notion of Jordan chains for analytic operator-valued functions.

\begin{defn} \label{defn:Jordan chain}
Let $X$ be a complex Banach space and $L: \mathbb{C} \to \mathcal{L}(X, X)$ an analytic operator-valued function. Given a complex number $\mu \in \mathbb{C}$, we say that an ordered set $x_0, \ldots, x_{k-1}$ of vectors in $X$ is a \textbf{Jordan chain of length $k$ for $L$ at $\mu$} if $x_0 \neq 0$ and 
\begin{equation} \label{eq:defn Jordan chain}
L(z) \left[ x_0 + (z-\mu) x_1 + \ldots + (z-\mu)^{k-1} x_{k-1} \right] = \mathcal{O} \left((z-\mu)^k\right). 
\end{equation}
The maximal length of a Jordan chain starting with $x_0$ is called the \textbf{rank} of $x_0$; the rank is said to be infinite if no maximum exists. 
\end{defn}

\begin{example}[cf. {\cite[p. 485]{KaashoekVL92}}]
Given a bounded linear operator $T: X \to X$, 
the usual notion of a Jordan chain for $T$ coincides with the notion of a Jordan chain for the analytic function
\[ L: \mathbb{C} \to \mathcal{L}(X, X), \qquad L(z) = zI - T. \]
Indeed, let $\mu \in \mathbb{C}$ be an eigenvalue of $T$ and let $x_0, \ldots, x_{k-1}$ be an associated Jordan chain, i.e.  
\begin{equation} \label{eq:Jordan chain operator}
T x_0 = \mu x_0, \quad Tx_1 = \mu x_1 + x_0, \quad \ldots \quad Tx_{k-1} = \mu x_{k-1} + x_{k-2}. 
\end{equation}
Then 
\begin{align*}
(zI - T) &\left[x_0 + (z-\mu) x_1 + \ldots + (z-\mu)^{k-1} x_{k-1} \right] \\
& = (z-\mu)x_0 + (z-\mu) \left[(z-\mu)x_1 - x_0 \right] + \ldots + (z-\mu)^{k-1} \left[(z-\mu)x_{k-1} - x_{k-2} \right] \\
& = (z-\mu)^k x_{k-1}. 
\end{align*}
So
\begin{equation} \label{eq:Jordan chain operator 2}
(zI - T) \left[x_0 + (z-\mu) x_1 + \ldots + (z-\mu)^{k-1} x_{k-1} \right] = \mathcal{O}((z-\mu)^k) 
\end{equation}
and $x_0, \ldots, x_{k-1}$ is a Jordan chain for the analytic function $z \mapsto zI - T$. Vice versa, suppose the vectors $x_0, \ldots, x_{k-1}$ satisfy \eqref{eq:Jordan chain operator 2}. Then evaluating the derivatives of \eqref{eq:Jordan chain operator 2} at $z = \mu$ yields the equalities \eqref{eq:Jordan chain operator} and thus $x_0, \ldots, x_{k-1}$ is a Jordan chain for the bounded operator $T$.  
\end{example}

If $x_0, \ldots, x_{k-1}$ is Jordan chain for $L$ at $\mu$, then \eqref{eq:defn Jordan chain} implies that $L(\mu) x_0 = 0$. Vice versa, if $x_0 \neq 0$ satisfies $L(\mu)x_0$, then $L(z)x_0 = \mathcal{O}(z-\mu)$ and hence $L$ has a Jordan chain (of at least length 1) at $\mu$ starting with $x_0$. So $L$ has a Jordan chain at $\mu$ starting with $x_0$ if and only $x_0$ is a non-zero element of the space
\[ \ker L(\mu) = \{x \in X \mid L(\mu)x =0 \}.  \]
We now consider the case in which the space $\ker L(\mu)$ is finite dimensional and all Jordan chains of $L$ at $\mu$ have finite rank. We pick a basis $x^0, \ldots, x^n$ of $\ker L(\mu)$ and for $1 \leq j \leq n$, we let $r_j$ be the rank of $x^j$. Then, if $x \neq 0$ is an element of $L(\mu)$, its rank has to be equal to one of the $r_j$. In particular, the set $\{r_1, \ldots, r_n\}$ does not depend on the choice of basis. We define the \textbf{algebraic multiplicity} of $\mu$ as the number 
\[ r_1 + \ldots + r_n. \]

The next lemma shows that the algebraic multiplicity is invariant under conjugation with analytic matrix-valued functions whose values are invertible operators. 

\begin{lemma}[{\cite[Proposition 1.2]{KaashoekVL92}, \cite[Proposition 5.1.1]{KaashoekVL20}}] \label{lem:equivalence Jordan chains} Given a complex Banach space $X$, let 
\[ L, M: \mathbb{C} \to \mathcal{L}(X, X) \]
be analytic operator-valued functions, and let 
\[ E, F: \mathbb{C} \to \mathcal{L}(X, X) \]
be analytic operator-valued functions whose values are invertible operators. Suppose that 
\[ M(z) = F(z) L(z) E(z) \]
for all $z \in \mathbb{C}$. Then, for $\mu \in \mathbb{C}$, the algebraic multiplicity of $L$ at $\mu$ equals the algebraic multiplicity of $M$ at $\mu$. 
\end{lemma}
\begin{proof}
We show that there is a one-to-one correspondence between Jordan chains for $L$ at $\mu$ and Jordan chains for $M$ at $\mu$; from there the claim follows. 

Let $x_0, \ldots, x_{k-1}$ be a Jordan chain for $L$ at $\mu$, i.e. 
\[ L(z) \left[x_0 + \ldots + (z-\mu)^{k-1} x_{k-1} \right] = \mathcal{O}\left((z-\mu)^{k}\right). \]
For $n \in \mathbb{N}$, let $y_n \in X$ be such that
\[ E(z)^{-1} \left[x_0 + \ldots + (z-\mu)^{k-1} x_{k-1} \right] = \sum_{n = 0}^\infty y_n (z-\mu)^n. \]
Then $y_0 \neq 0$ and 
\[ E(z) \left[y_0 + \ldots + (z-\mu)^{k-1} y_{k-1} \right] =x_0 + \ldots + (z-\mu)^{k-1} x_{k-1} - E(z) \sum_{n = k}^\infty y_n (z-\mu)^n \]
so $y_0, \ldots, y_{k-1}$ satisfy
\begin{align*}
M(z) \left[y_0 + \ldots + (z-\mu)^{k-1} y_{k-1} \right] &= F(z) L(z) E(z) \left[y_0 + \ldots + (z-\mu)^{k-1} y_{k-1} \right] \\
&= F(z) L(z) \left[x_0 + \ldots + (z-\mu)^{k-1} x_{k-1}\right]  \\ & \qquad \qquad - F(z) L(z) E(z) \sum_{n = k}^\infty y_n (z-\mu)^n \\
& = \mathcal{O}\left((z-\mu)^k \right).
\end{align*}
So $y_0, \ldots, y_{k-1}$ is a Jordan chain for $M$ at $\mu$. Vice versa, every Jordan chain $y_0, \ldots, y_{k-1}$ of $M$ at $\mu$ induces a Jordan chain for $L$ at $\mu$. So there is a one-to-one correspondence between Jordan chains for $M$ at $\mu$ and Jordan chains for $L$ at $\mu$. In particular, the algebraic multiplicity of $L$ at $\mu$ equals the algebraic multiplicity of $M$ at $\mu$.
\end{proof}

We are now ready to make precise how a characteristic matrix function, as defined in Definition \ref{defn:cm}, captures the spectral information of a bounded linear operator: 

\begin{lemma}[{\cite[Theorem 5.2.2]{KaashoekVL20}}] \label{lem:eigenvalues cm}
Let $X$ be a complex Banach space, $T: X \to X$ a bounded linear operator and $\Delta: \mathbb{C} \to \mathbb{C}^{n \times n}$ a characteristic matrix function for $T$. Let $\mu \in \mathbb{C} \backslash \{0 \}$, then 
\begin{enumerate}
\item $\mu^{-1} \in \sigma_{pt}(T)$ if and only if $\det \Delta(\mu) = 0$, i.e. 
\[ \sigma_{pt}(T) \backslash \{0 \}  = \{ \mu^{-1} \in \mathbb{C} \mid \det \Delta(\mu) = 0 \}. \]
\item If $\mu^{-1} \in  \sigma_{pt}(T)$, then the geometric multiplicity of $\mu^{-1}$ as an eigenvalue of $T$ equals the dimension of the space
\[ \ker \Delta(\mu) = \{x \in \mathbb{C}^n \mid \Delta(\mu)x = 0 \}. \]
\item If $\mu^{-1} \in  \sigma_{pt}(T)$, then the algebraic multiplicity of $\mu^{-1}$ as an eigenvalue of $T$ equals the order of $\mu$ as a root of
\begin{equation} \label{eq:cm lemma}
 \det \Delta(z) = 0. 
\end{equation} 
\end{enumerate}
\end{lemma}
\begin{proof}
Let $\mu \neq 0$, then $\mu^{-1}$ is an eigenvalue of $T$ if and only if $\mu^{-1} I - T = \mu^{-1} \left(I - \mu T \right)$ has a non-trivial kernel, i.e. if and only if $I - \mu T$ has a non-trivial kernel. 
We first show that there is a one-to-one correspondence between kernel vectors of $I - \mu T$ and kernel vectors of $\Delta(\mu)$. This then implies the first two statements of the lemma. 

We write
\begin{equation} \label{eq:L and M kernels}
M(z) = \begin{pmatrix}
\Delta(z) & 0 \\
0 & I_{X}
\end{pmatrix}, \qquad L(z) = \begin{pmatrix}
I_{\mathbb{C}^n} & 0 \\
0 & I - zT
\end{pmatrix},
\end{equation}
then the kernels of the operators $L(\mu), \ M(\mu)$ are given by
\[\ker M(\mu) = \ker \Delta(\mu) \oplus \{0 \}, \qquad \ker L(\mu) = \{0 \} \oplus \ker (I - \mu T). \]
Since $\Delta$ is a characteristic matrix function for $T$, there exists analytic functions $E, F: \mathbb{C} \to \mathcal{L}(\mathbb{C}^n \oplus X, \mathbb{C}^n \oplus X)$ so that $E(z), F(z)$ are invertible operators for all $z \in \mathbb{C}$ and such that $M(z) = F(z) L(z) E(z)$ for all $z \in \mathbb{C}$.  In particular, the operator $E(\mu)$ maps the space $ \ker M(\mu)$ in a one-to-one way to the space $\ker L(\mu)$. This implies that the map
\begin{align*}
\ker \Delta(\mu) \to \ker (I - \mu T), \qquad 
c \mapsto (0, I_X) E(\mu) \begin{pmatrix}
c \\ 0
\end{pmatrix}
\end{align*}
with inverse
\begin{align*}
\ker (I-\mu T) \to \ker \Delta(\mu), \qquad 
x \mapsto (I_{\mathbb{C}^n}, 0) E(\mu)^{-1} \begin{pmatrix}
0 \\ x
\end{pmatrix}
\end{align*}
is a bijection. So there is a one-to-one correspondence between elements of $\ker (I-\mu T)$ and elements of $\ker \Delta(\mu)$, which proves the first two statements of the lemma. 

\medskip

To prove the third statement of the lemma, we first show that for any number $\mu \in \mathbb{C}$, the algebraic multiplicity of $\Delta$ at $\mu$ equals the order of $\mu$ as a root of the equation $\det \Delta(z) = 0$. 
To do so, we bring $\Delta$ in \emph{local Smith form}: given $\mu \in \mathbb{C}$, there exists analytic functions $G, H$, whose values are invertible matrices, and unique non-negative integers $r_1, \ldots, r_N$ such that 
\begin{subequations}
\begin{equation} \label{eq:Smith form 1}
\Delta(z) = G(z) D(z) H(z) 
\end{equation}
with
\begin{equation} \label{eq:Smith form 2}
D(z) = \begin{pmatrix}
(z-\mu)^{r_1} & 0 & \ldots & 0 \\
0 & (z-\mu)^{r_2} & \ldots & 0 \\
0 & 0 & \ddots & 0 \\
0 & 0 & \ldots & (z-\mu)^{r_N}
\end{pmatrix},
\end{equation}
\end{subequations}
see \cite[Section 1]{KaashoekVL92}. The algebraic multiplicity of $D$ at $\mu$ equals $r_1 + \ldots + r_N$; therefore, Lemma \ref{lem:equivalence Jordan chains} implies that the algebraic multiplicity of $\Delta$ at $\mu$ equals $r_1 + \ldots + r_N$ as well. On the other hand, by \eqref{eq:Smith form 1}--\eqref{eq:Smith form 2} we can write $\det \Delta$ as
\[ \det \Delta(z) = \det \left(G(z) \right) (z-\mu)^{r_1} \ldots (z-\mu)^{r_N} \det \left(H(z)\right). \]
Since $G(z), H(z)$ are invertible matrices, the order of $\mu$ as a root of $\det \Delta(z) = 0$ is also given by $r_1 + \ldots + r_N$. We conclude that the algebraic multiplicity of $\Delta$ at $\mu$ equals the order of $\mu$ as a root of $\det \Delta(z) = 0$. 

Now let $\mu^{-1} \in \sigma_{pt}(T)$. Then by equality \eqref{eq:defn cm} and Lemma \ref{lem:equivalence Jordan chains}, the algebraic multiplicity of $\mu^{-1}$ as an eigenvalue of $T$ equals the algebraic multiplicity of $\Delta$ at $\mu$; by the previous step, this equals the order of $\mu$ as a root of $\det \Delta(z) = 0$. We conclude that the algebraic multiplicity of $\mu^{-1}$ as an eigenvalue of $T$ equals the order of $\mu$ as a root of $\det \Delta(z) = 0$, as claimed. 
\end{proof}

The next theorem from \cite{KaashoekVL20} gives a sufficient condition for a bounded linear operator to have a characteristic matrix function. The proof of this theorem is beyond the scope of this article and hence we state the theorem without proof. 

\begin{theorem}[{\cite[Theorem 6.1.1]{KaashoekVL20}}] \label{thm: CM V + R}  Let $X$ be a complex Banach space and $T: X \to X$ a bounded linear operator. Assume that $T$ is of the form $T = V + R$
with
\begin{enumerate}
\item $V: X \to X$ a Volterra operator, i.e. $V$ is compact and $\sigma(V) \subseteq \{0 \}$;
\item $R$ an operator of finite rank $n \in \mathbb{N}$. 
\end{enumerate}
Decompose $R$ as $R = DC$ where
\[C: X \to \mathbb{C}^n \qquad \mbox{and} \qquad D: \mathbb{C}^n \to X. \]
Then the matrix-valued function
\[ \Delta(z) = I_{\mathbb{C}^n} - zC(I-zV)^{-1} D \]
is a characteristic matrix function for $T$. 
\end{theorem}

Note that, since $V$ is Volterra, the resolvent map $z \mapsto (I-zV)^{-1}$ is analytic on $\mathbb{C}$ and hence $z \mapsto \Delta(z)$ is an analytic matrix-valued function. 

\section{Characteristic matrices for DDE with symmetries} \label{sec:proof} 
In this section we prove this article's main result Theorem \ref{thm:main result}. 
We prove this theorem by writing the operator \eqref{eq:reduced mon op} as the sum of a Volterra operator and a finite rank operator; we then apply Theorem \ref{thm: CM V + R} to obtain a characteristic matrix function. 

In \cite[Section 11.4]{KaashoekVL20}, the authors consider a DDE with a periodic solution; they do \emph{not} assume any symmetry relations on  the DDE, but do assume that the period of the periodic solution is equal to the time delay. 
In this setting, they construct a characteristic matrix function for the monodromy operator. The difference between Theorem \ref{thm:main result} presented here and the result in \cite[Section 11.4]{KaashoekVL20} is the following: we realize that 
in a \emph{symmetric} setting, the operator \eqref{eq:reduced mon op} has a characteristic matrix function if the time shift of the spatio-temporal symmetry is equal to the time delay, whereas in the non-symmetric case considered in \cite[Section 11.4]{KaashoekVL20} the monodromy operator has a characteristic matrix function if the period of periodic solution is equal to the time delay. The proof of Theorem \ref{thm:main result} is similar in spirit to the arguments in \cite[Section 11.4]{KaashoekVL20}, but we additionally exploit the equivariance of the considered system. 

This section is structured as follows: in Subsection \ref{subsec:linear case} we first consider a linear, time-dependent DDE whose coefficients satisfy a spatio-temporal relation; for this DDE we construct a characteristic matrix function. The proof of Theorem \ref{thm:main result} then follows in Subsection \ref{subsec:proof}.

\subsection{Linear, time-dependent DDE with spatio-temporal symmetry}
\label{subsec:linear case}

We consider the initial value problem
\begin{subequations}
\begin{align} \label{eq:dde cm}
\begin{cases}
\dot{y}(t) = A(t)y(t) + B(t) y(t-\tau), &\qquad \mbox{for } t \geq s; \\
y(s+t) = \varphi(t) &\qquad \mbox{for } t \in [-\tau, 0]. 
\end{cases}
\end{align}
with time delay $\tau > 0$ and initial condition $\phi \in C \left([-\tau, 0], \mathbb{R}^N\right)$ at time $s \in \mathbb{R}$. We make the following assumptions on system \eqref{eq:dde cm}:

\begin{hyp}\label{hyp:symmetries linear} \hfill
\begin{enumerate}
\item the functions $A, B: \mathbb{R} \to \mathbb{R}^{N \times N}$ are $C^2$;
\item there exists an invertible matrix $h \in \mathbb{R}^{N \times N}$ such that
\begin{equation} \label{eq:periodicity}
h A(t) h^{-1} = A(t+\tau), \qquad h B(t) h^{-1} = B(t+\tau)
\end{equation}
for all $t \in \mathbb{R}$. 
\end{enumerate}
\end{hyp}
\end{subequations}

\noindent We stress that the time shift $\tau$ in equation \eqref{eq:periodicity} is the same as the time delay of the DDE \eqref{eq:dde cm}. 
So the coefficients $A, B$ satisfy some spatio-temporal relation with time shift equal to the time delay of \eqref{eq:dde cm}. 
Under this hypothesis, we construct a characteristic matrix function $\Delta$ for the operator $h^{-1} U(\tau, 0)$, where $U(t, s), \ t \geq s$ is the family of solution operators of \eqref{eq:dde cm}. We give an explicit expression for $\Delta$ in terms of solutions of the family of ODE
\begin{equation} \label{eq:reduced ODE}
\dot{y}(t) = \left[A(t) + z \cdot B(t) h^{-1} \right] y(t)
\end{equation}
with $z \in \mathbb{C}$. 
To arrive at this expression for $\Delta$, we make the following intermediate steps: 

\begin{enumerate}
\item We show that the symmetry relations \eqref{eq:periodicity} imply symmetry relations on the fundamental solution of the ODE \eqref{eq:reduced ODE} (Lemma \ref{lem:fund sol});
\item We give an explicit expression for $h^{-1}U(\tau, 0)$ (Lemma \ref{lem:expression T}) and write 
$h^{-1}U(\tau, 0) = V +R$, with $V$ an integral operator and $R$ an operator of finite rank; 
\item We show that the integral operator $V$ is in fact a Volterra operator (Lemma \ref{lem:Volterra});
\item We apply Theorem \ref{thm: CM V + R} to find a characteristic matrix for $h^{-1} U(\tau, 0)$ (Proposition \ref{prop: cm delay equals period}). 
\end{enumerate}

In the case where the coefficients $A, B$ of \eqref{eq:reduced ODE} are periodic, i.e. when $A(t+p) = A(t), \ B(t+p) = B(t)$ for some $p > 0$, the fundamental solution $F(t, z)$ of \eqref{eq:reduced ODE} satisfies the additional relation
\begin{equation}\label{eq:periodicity F}
F(t+p, z) F(s+p, z)^{-1} = F(t, z) F(s, z)^{-1}.
\end{equation}
Indeed, the matrix-valued function $t \mapsto F(t+p,z)F(s+p, z)^{-1}$ satisfies the ODE
\begin{align*}
\frac{d}{dt} F(t+p,z)F(s+p, z)^{-1} &= \left[A(t+p) + B(t+p)h^{-1} \right] F(t+p,z)F(s+p, z)^{-1} \\
&= \left[A(t)+B(t)h^{-1} \right] F(t+p, z) F(s+p, z)^{-1}
\end{align*}
with initial condition $F(t+p,z)F(s+p, z)^{-1} = I_{\mathbb{C}^N}$ for $t = s$. So uniqueness of solutions implies \eqref{eq:periodicity F}. Similarly, the symmetry relation \eqref{eq:periodicity} on the coefficients $A, B$ induce symmetry relations on the fundamental solution $F(t, z)$, as we make precise in the following lemma:

\begin{lemma} \label{lem:fund sol} Consider functions $A, B: \mathbb{R} \to \mathbb{R}^{N \times N}$ satisfying Hypothesis \ref{hyp:symmetries linear}. 
For $z \in \mathbb{C}$, let $F(t, z)$ be the fundamental solution of the ODE
\[ \dot{y}(t) = \left[A(t) + z \cdot B(t) h^{-1} \right] y(t) \]
with $F(0, z) = I_{\mathbb{C}^N}$. Then it holds that 
\begin{equation} \label{eq:fund sol red ode}
h F(t, z) F(s, z)^{-1} h^{-1} = F(t+\tau, z)F(s+\tau, z)^{-1}
\end{equation}
for all $t \geq s$. 

In particular, if $Y_A(t)$ is the fundamental solution of the ODE
\[ \dot{y}(t) = A(t) y(t) \]
with $Y(0) = I_{\mathbb{C}^N}$, then 
\begin{equation} \label{eq:fund sol A}
h Y_A(t) Y_A(s)^{-1} h^{-1} = Y_A(\tau+t)Y_A(\tau+s)^{-1}.
\end{equation}
\end{lemma} 
\begin{proof}
The matrix-valued function $t \mapsto hF(t, z) F(s, z)h^{-1}$ satisfies the ODE
\begin{align*}
\frac{d}{dt}  hF(t, z) F(s, z)h^{-1} &= h \left[ A(t) + z \cdot B(t) h^{-1} \right]  F(t, z) F(s, z)h^{-1} \\
& = \left[A(t+\tau) + z \cdot B(t+\tau) h^{-1} \right] \left[ h F(t, z) F(s, z)^{-1} h^{-1} \right]
\end{align*} 
with initial condition $ hF(t, z) F(s, z)h^{-1} = I_{\mathbb{C}^N}$ for $t = s$. Similarly, the matrix-valued function $t \mapsto F(t+\tau,z)F(s+\tau, z)^{-1}$ satisfies the ODE
\begin{align*}
\frac{d}{dt} F(t+\tau,z)F(s+\tau, z)^{-1} = \left[A(t+\tau) + z \cdot B(t+\tau) h^{-1} \right] F(t+\tau,z)F(s+\tau, z)^{-1}
\end{align*}
with initial condition $F(t+\tau,z)F(s+\tau, z)^{-1} = I_{\mathbb{C}_N}$ for $t = s$. Uniqueness of solutions now implies the relation \eqref{eq:fund sol red ode}. 

Since $F(t, z) = Y_A(t)$ when $B \equiv 0$, the equality \eqref{eq:fund sol red ode} implies the equality \eqref{eq:fund sol A}. 
\end{proof}

We next give an explicit expression for the operator $h^{-1} U(\tau, 0)$.

\begin{lemma} \label{lem:expression T}
Consider the DDE \eqref{eq:dde cm} satisfying Hypothesis \ref{hyp:symmetries linear} and let $U(t, s), \ t \geq s$ be the family of solution operators of \eqref{eq:dde cm}. Moreover, let $Y_A(t)$ be the fundamental solution of the ODE
\[ \dot{y}(t) = A(t) y(t) \]
with $Y_A(0) = I_{\mathbb{C}^N}$. 
Then the operator 
\[ h^{-1} U(\tau, 0): C \left([-\tau,0], \mathbb{R}^N\right) \to  C \left([-\tau,0], \mathbb{R}^N\right) \]
is given by
\begin{equation}
(h^{-1} U(\tau,0) \phi)(\theta) = h^{-1} Y_A(\tau+\theta) \phi(0) + \int_{-\tau}^\theta  Y_A(\theta) Y_A(s)^{-1} B(s) h^{-1} \phi(s) ds.  \label{eq: reduced monodromy cm}
\end{equation}
\end{lemma}
\begin{proof}
For $s = 0$ and $t \in [0, \tau]$, the initial value problem \eqref{eq:dde cm} becomes
\[ \dot{y}(t) =A(t) y(t) + B(t) \phi(t-\tau), \qquad y(0) = \phi(0),\]
which we solve by the Variation of Constants formula as
\begin{equation} \label{eq: VoC cm}
y(t) =Y_A(t) \phi(0) + \int_{0}^{t} Y_A(t) Y_A(s)^{-1} B(s) \phi(s-\tau) ds.
\end{equation} 
With $t = \tau + \theta ,\ \theta \in [-\tau, 0]$, \eqref{eq: VoC cm} becomes
\begin{align*}
y(\tau+\theta) &= Y_A(\tau+\theta) \phi(0) + \int_{0}^{\tau+\theta} Y_A(\tau+\theta) Y_A(s)^{-1} B(s) \phi(s-\tau) ds \\
& = Y_A(\tau+\theta) \phi(0) + \int_{-\tau}^\theta Y_A(\tau+\theta) Y_A(\tau+s)^{-1} B(\tau+s) \phi(s) ds \\
& =Y_A(\tau+\theta) \phi(0) + \int_{-\tau}^\theta h Y_A(\theta) Y_A(s)^{-1} h^{-1} h B(s) h^{-1} \phi(s) ds
\end{align*}
where in the last step we used \eqref{eq:periodicity} and \eqref{eq:fund sol A}. 
So $(h^{-1} U(\tau, 0) \phi)(\theta) := h^{-1} y(\tau+\theta)$ is given by
\[ (h^{-1} U(\tau, 0) \phi)(\theta) =h^{-1} Y_A(\tau+\theta) \phi(0) + \int_{-\tau}^\theta  Y_A(\theta) Y_A(s)^{-1} B(s) h^{-1} \phi(s) ds, \]
which proves the lemma. 
\end{proof}

To apply Theorem \ref{thm: CM V + R}, we first complexify the operator $h^{-1}U(\tau, 0)$ in \eqref{eq: reduced monodromy cm} via a canonical procedure as detailed in, for example, \cite[Chapter 3.7]{delayequations}. However, we do not make the complexification explicit in notation, i.e. we write  $h^{-1}U(\tau, 0)$ both for the real operator on the real Banach space $C\left([-\tau, 0], \mathbb{R}^N\right)$ and the complexified operator on the complex Banach space $C \left([-\tau, 0], \mathbb{C}^N\right)$. 
We then decompose the complex operator $h^{-1} U(\tau, 0)$ as
\begin{equation} \label{eq: decomposition cm}
h^{-1} U(\tau, 0) = V + R
\end{equation}
with the (suggestive) notation
\begin{subequations}
\begin{align}  
\begin{split}
V: X \to X, &\qquad (V \phi)(\theta) = \int_{-\tau}^\theta Y_A(\theta) Y_A(s)^{-1} B(s) h^{-1} \phi(s) ds,  \label{eq: V cm}
\end{split} \\
\begin{split}
R: X \to X, &\qquad (R \phi)(\theta) =h^{-1} Y_A(\tau+\theta) \phi(0).   \label{eq: R cm}
\end{split}
\end{align}
\end{subequations}
and complex Banach space $X$ given by
\[X = C \left([-\tau, 0], \mathbb{C}^N \right). \] 
We next prove that the integral operator \eqref{eq: V cm} is in fact a Volterra operator.  

\begin{lemma}\label{lem:Volterra}
The operator $V$ defined in \eqref{eq: V cm} is Volterra, i.e. $V$ is compact and $\sigma(V) \subseteq \{0 \}$. 
\end{lemma}
\begin{proof}
We first prove that $\sigma_{pt}(V) \subseteq \{0 \}$. To do so, fix $z \in \mathbb{C} \backslash \{0 \}$ and let $\phi \in X$ be such that $V \phi = z \phi$, i.e.
\begin{equation} \label{eq:tussenstap 1}
\int_{-\tau}^\theta Y_A(\theta) Y_A(s)^{-1} B(s) h^{-1} \phi(s) ds = z \phi(\theta).
\end{equation}
Equality \eqref{eq:tussenstap 1} implies that $\phi(-\tau) = 0$. Moreover, since the left hand side of \eqref{eq:tussenstap 1} is $C^1$, the right hand side is $C^1$ as well;  differentiating both sides with respect to $\theta$ gives
\begin{equation} \label{eq:tussenstap 2}
A(\theta) \int_{-\tau}^\theta Y_A(\theta) Y_A(s)^{-1} B(s) h^{-1} \phi(s) ds + B(\theta) h^{-1} \phi(\theta) = z \phi'(\theta).
\end{equation}
Equality \eqref{eq:tussenstap 1} also implies that
\[ A(\theta) \int_{-\tau}^\theta Y_A(\theta) Y_A(s)^{-1} B(s) h^{-1} \phi(s) ds = z A(\theta) \phi(\theta) \]
and substituting this into \eqref{eq:tussenstap 2} gives
\[ \mu A(\theta) \phi(\theta) + B(\theta) h^{-1} \phi(\theta) = z \phi'(\theta).\]
So $\phi$ satisfies the initial value problem 
\begin{align*}
\begin{cases}
\phi'(\theta) = A(\theta) \phi(\theta) + z^{-1} B(\theta) h^{-1} \phi(\theta), \qquad \theta \in [-\tau, 0], \\
\phi(-\tau) = 0,
\end{cases}
\end{align*}
which implies that $\phi \equiv 0$. We conclude that $z \in \mathbb{C} \backslash \{0 \}$ is not an eigenvalue of $V$, and thus that $\sigma_{pt}(V) \subseteq \{0 \}$. 

If $ \phi \in C \left([-\tau, 0], \mathbb{R}^N \right)$, then \eqref{eq: V cm} implies that $V \phi \in C^1$ and hence by the Arzel{\`a}-Ascoli theorem $V$ is compact. This implies that the non-zero spectrum of $V$ consists of eigenvalues.  Since we already showed that $\sigma_{pt}(V) \subseteq \{0 \}$, we conclude that $\sigma(V) \subseteq \{0 \}$. So $V$ is a compact operator and $\sigma(V) \subseteq \{0 \}$, which proves the claim. 
\end{proof}

We are now ready to use Theorem \ref{thm: CM V + R} and give an explicit characteristic matrix function for the operator $h^{-1} U(\tau, 0)$:

\begin{prop} \label{prop: cm delay equals period}
Consider the DDE \eqref{eq:dde cm} satisfying Hypothesis \ref{hyp:symmetries linear}; let $U(t, s), \ t \geq s$ be the family of solution operators of \eqref{eq:dde cm}. 
Moreover, for $z \in \mathbb{C}$, let $F(t, z)$ be the fundamental solution of the ODE
\[ \dot{y}(t) = \left[A(t) + z \cdot B(t) h^{-1} \right] y(t) \]
with $F(0, z) = I_{\mathbb{C}^N}$. Then the matrix-valued function
\begin{equation}\label{eq:cm delay equals period}
\Delta(z) = I_{\mathbb{C}^N} - z h^{-1}  F(\tau, z)
\end{equation}
is a characteristic matrix for the operator 
\begin{equation} \label{eq:red mon op 2}
h^{-1}  U(\tau, 0). 
\end{equation}
\end{prop}
\begin{proof} 
We divide the proof into two steps: 

\medskip
\noindent \emph{Step 1:} The finite rank operator $R$ in \eqref{eq: R cm} factorizes as $R = DC$ with 
\begin{subequations}
\begin{align}
\begin{split}
C: X \to \mathbb{C}^N, &\qquad C \phi = \phi(0) \label{eq: C cm} 
\end{split} \\
\begin{split}
D:\mathbb{C}^N \to X, &\qquad (D u)(\theta) = h^{-1} Y_A(\tau+\theta) u. \label{eq: B cm}
\end{split}
\end{align}
\end{subequations}
For $V$ as in \eqref{eq: V cm} and $D$ as in \eqref{eq: B cm}, we now give an explicit expression for $\left(I-zV\right)^{-1} D$. To that end, fix $z \in \mathbb{C}$ and $u \in \mathbb{C}^N$; let $\phi \in X$ be the unique element such that $(I-zV)^{-1} Du = \phi$, i.e. $\phi$ satisfies
\begin{equation} \label{eq:tussenstap 3}
\phi(\theta) = h^{-1} Y_A(\tau+\theta) u + z \int_{-\tau}^\theta Y_A(\theta) Y_A(s)^{-1} B(s) h^{-1} \phi(s) ds.
\end{equation}
Equality \eqref{eq:tussenstap 3} implies that $\phi(-\tau) = h^{-1} u$. Moreover, since the right hand side of \eqref{eq:tussenstap 3} is $C^1$, the left hand side is $C^1$ as well; differentiating both sides with respect to $\theta$ gives
\begin{align}
\phi'(\theta) &= h^{-1} A(\tau+\theta) Y_A(\tau+\theta) u + z A(\theta) \int_{-\tau}^\theta Y_A(\theta) Y_A(s)^{-1} B(s) h^{-1} \phi(s) ds + z B(\theta) h^{-1} \phi(\theta) \\
& = A(\theta) h^{-1} Y_A(\tau+\theta) u + z A(\theta) \int_{-\tau}^\theta Y_A(\theta) Y_A(s)^{-1} B(s) h^{-1} \phi(s) ds + z B(\theta) h^{-1} \phi(\theta) \label{eq:tussenstap 4}
\end{align}
where in the last step we used \eqref{eq:periodicity}. Equality \eqref{eq:tussenstap 3} also implies that 
\[ z A(\theta) \int_{-\tau}^\theta Y_A(\theta) Y_A(s)^{-1} B(s) h^{-1} \phi(s) ds = A(\theta) \phi(\theta) - A(\theta) h^{-1} Y_A(\tau+\theta) u \]
and substituting this into \eqref{eq:tussenstap 4} gives that
\[ \phi'(\theta) = A(\theta) \phi(\theta) + z \cdot B(\theta) h^{-1} \phi(\theta). \]
So  $\phi$ satisfies the initial value problem
\begin{align*}
\begin{cases}
\phi'(\theta)  &= A(\theta) \phi(\theta) + z \cdot B(\theta) h^{-1} \phi(\theta), \quad \theta \in [-\tau, 0], \\
\phi(-\tau) &= h^{-1} u
\end{cases}
\end{align*}
which implies that $\phi(\theta) = F(\theta, z) F(-\tau, z)^{-1} h^{-1} u$. Equality \eqref{eq:fund sol red ode} with $t = \theta$ and $s = -\tau$ implies that
\[ F(\theta, z) F(-\tau, z)^{-1} h^{-1} = h^{-1} F(\tau+\theta, z) \]
and hence
\[ \phi(\theta) = h^{-1} F(\tau+\theta, z) u. \]
So we conclude that
\begin{equation} \label{eq:simplification}
\left((I-zV)^{-1} D \right)(\theta) = h^{-1} F(\tau+\theta, z).
\end{equation}

\noindent
\emph{Step 2:} We now prove the statement of the proposition. The operator $h^{-1} U(\tau, 0)$ decomposes as
\[ h^{-1} U(\tau, 0) = V + R, \]
with $V$ defined in \eqref{eq: V cm} and $R$ defined in \eqref{eq: R cm}. The operator $R$ is a finite rank operator; by Lemma \ref{lem:Volterra}, the operator $V$ is a Volterra operator.
Therefore, if we let $D, C$ be as in \eqref{eq: C cm}--\eqref{eq: B cm}, Theorem \ref{thm: CM V + R}  implies that 
\[ \Delta(z) =  I_{\mathbb{C}^N} - z C(I-zV)^{-1} D \]
is a characteristic matrix for $h^{-1} U(\tau,0)$. Equality \eqref{eq:simplification} implies that 
\[ C(I-zV)^{-1} D = h^{-1} F(\tau, z) \]
and hence 
\[ \Delta(z) = I_{\mathbb{C}^N} - z h^{-1}  F(\tau, z) \]
is a characteristic matrix for $h^{-1} U(\tau, 0)$, as claimed. 
\end{proof}

\subsection{Proof of Theorem \ref{thm:main result}}
\label{subsec:proof}
Theorem \ref{thm:main result} now follows from Proposition \ref{prop: cm delay equals period}: 

\begin{proof}[Proof of Theorem \ref{thm:main result}]
Define 
\[ A(t) = \partial_1 f(x_\ast(t), x_\ast(t-\tau)), \qquad B(t) = \partial_2 f(x_\ast(t), x_\ast(t-\tau)). \]
We show that the coefficients $A, B$ satisfy Hypothesis \ref{hyp:symmetries linear}. 
By assumption, the periodic solution $x_\ast$ has a spatio-temporal symmetry $h \in H_\ast$ with  
\[ hx_\ast(t) = x_\ast(t+\tau). \]
Moreover, since $f: \mathbb{R}^N \times \mathbb{R}^N \to \mathbb{R}^N$ satisfies the equivariance relation \eqref{eq:equivariance}, it holds that
\begin{align*}
\partial_i f(h x, \gamma y) h &= h \partial_i f(x, y) \\
\end{align*}
for all $x, y \in \mathbb{R}^N$ and $i = 1,2$. So it in particular holds that 
\begin{align*}
A(t+\tau) h &= \partial_1 f(x_\ast(t+\tau), x_\ast(t)) h  \\
& = \partial_1 f(hx_\ast(t), hx_\ast(t-\tau)) h \\
& = h \partial_1 f(x_\ast(t), x_\ast(t-\tau)) = h A(t)
\end{align*}
and similarly
\begin{align*}
B(t+\tau) h &= \partial_2 f(x_\ast(t+\tau), x_\ast(t)) h  \\
& = \partial_2 f(hx_\ast(t), hx_\ast(t-\tau)) h \\
& = h \partial_2 f(x_\ast(t), x_\ast(t-\tau)) = h B(t).
\end{align*}
So the coefficients $A, B$ satisfy Hypothesis \ref{hyp:symmetries linear}; therefore Proposition \ref{prop: cm delay equals period} implies Theorem \ref{thm:main result}. 
\end{proof}

We considered the system \eqref{eq:dde cm}--\eqref{eq:periodicity} with in the back of our mind the linearized DDE \eqref{eq:linearized dde thm}. However, the equations \eqref{eq:dde cm}--\eqref{eq:periodicity} also cover the special case $h = I$. In this case, the equation \eqref{eq:dde cm} has periodic coefficients with period equal to the time delay, and the operator \eqref{eq:red mon op 2} is the monodromy operator. So in this case, an application Proposition \ref{prop: cm delay equals period} gives a characteristic matrix for the monodromy operator, and we recover the result from \cite[Section 11.4]{KaashoekVL20}:

\begin{theorem}[{cf. \cite[Section 11.4]{KaashoekVL20}} ] \label{thm:periodicity}
Consider the DDE
\begin{equation} \label{eq:dde periodic}
\dot{x}(t) = f(x(t), x(t-\tau))
\end{equation}
with $f: \mathbb{R}^N \times \mathbb{R}^N \to \mathbb{R}^N$ a $C^2$ function and with time delay $\tau > 0$. Assume that system \eqref{eq:dde periodic} has a periodic solution $x_\ast$ with period $\tau$, i.e.
\[x_\ast(t+\tau) = x_\ast(t).\]
Let $U(t, s), \ t \geq s$ be the family of solution operators associated to the linearized DDE
\[ \dot{y}(t) = \partial_1 f(x_\ast(t), x_\ast(t-\tau)) y(t) + \partial_2 f(x_\ast(t), x_\ast(t-\tau))y(t-\tau).\]
For $z \in \mathbb{C}$, let $F(t,z)$ be the fundamental solution of the ODE
\[ \dot{y}(t) = \left[ \partial_1 f(x_\ast(t), x_\ast(t-\tau)) + z \partial_2 f(x_\ast(t), x_\ast(t-\tau)) \right] y(t) \]
with $F(0, z) = I_{\mathbb{C}^N}$. Then the analytic function
\[ \Delta(z) = I_{\mathbb{C}^N} - z F(\tau, z) \]
is a characteristic matrix function for the monodromy operator
\[ U(\tau,0). \]
\end{theorem}
\begin{proof}
Define 
\[ A(t) := \partial_1 f(x_\ast(t), x_\ast(t-\tau)), \qquad B(t) : = \partial_2 f(x_\ast(t), x_\ast(t-\tau)), \]
then it holds that
\[  A(t+\tau) = A(t), \qquad B(t+\tau) = B(t). \]
So the coefficients $A, B$ satisfy Hypothesis \ref{hyp:symmetries linear} with $h = I_{\mathbb{C}^N}$. Therefore Proposition \ref{prop: cm delay equals period} implies the statement of the theorem. 
\end{proof}

\section{Applications to delayed feedback control}
\label{sec:applications}

In \cite{Pyragas92}, Pyragas introduced a delayed feedback method (now known as \textbf{Pyragas control}) that aims to stabilize periodic orbits
of the ordinary differential equation
\begin{equation} \label{eq:ode pyragas}
\dot{x}(t) = F(x(t)), \qquad x(t) \in \mathbb{R}^N. 
\end{equation}
The feedback term introduced by Pyragas measures 
the difference between the current state and the state time $\tau$ ago, and feeds this difference (multiplied by a matrix) back into the system. Concretely the system with feedback control becomes
\begin{equation} \label{eq:pyragas}
\dot{x}(t) = F(x(t)) + K \left[x(t) -x(t-\tau) \right]
\end{equation}
with time delay $\tau > 0$ and matrix $K \in \mathbb{R}^{N \times N}$. 
If now $x_\ast(t)$ is a $\tau$-periodic solution of \eqref{eq:ode pyragas}, then it is also a solution of \eqref{eq:pyragas}. However, the overall dynamics of the systems with and without feedback are different, and it is possible that $x_\ast$ is an unstable solution of \eqref{eq:ode pyragas} but a stable solution of \eqref{eq:pyragas}.  

We can determine whether $x_\ast$ is a stable solution of \eqref{eq:pyragas} by computing the eigenvalues of the monodromy operator
\[ U(\tau, 0) \]
where  $U(t, s), \ t \geq s$ is the family of solution operators of the DDE
\[ \dot{y}(t) = F'(x_\ast(t)) y(t) + K \left[y(t) - y(t-\tau)\right]. \]
The results in \cite{KaashoekVL92}, \cite[Section 11.4]{KaashoekVL20} (cf. Theorem \ref{thm:periodicity} in this article) give a characteristic matrix of the monodromy operator $U(\tau,0)$ in terms of solutions of the ODE
\[ \dot{y}(t) =  F'(x_\ast(t)) y(t) + K \left[1-z\right] y(t). \]
The explicit expression for a characteristic matrix of $U(\tau,0)$ has for example been used in the context of feedback control of a Hamiltonian system \cite{Fiedler20} and in studying the behaviour of the control scheme \eqref{eq:pyragas} as the delay $\tau$ goes to infinity \cite{Sieber13}. 

\medskip
\textbf{Equivariant Pyragas control} \cite{Fiedler10} adapts the Pyragas feedback scheme so that the feedback term vanishes on a periodic orbit with a specific spatio-temporal pattern. 
More precisely, suppose that 
\begin{itemize}
\item \eqref{eq:ode pyragas} is equivariant with respect to a compact symmetry group $\Gamma \subseteq \mathrm{GL}(N, \mathbb{R})$;
\item \eqref{eq:ode pyragas} has a periodic solution $x_\ast$ with minimal period $p > 0$;
\item $x_\ast$ is a discrete wave and $h \in H_\ast$ is a spatio-temporal symmetry of $x_\ast$, i.e.
\[ hx_\ast(t) = x_\ast(t+\Theta(h)p) \]
for some $\Theta(h) \in [0, 1)$. 
\end{itemize}
Then the periodic solution $x_\ast$ is also a solution of the feedback system 
\begin{equation} \label{eq:equivariant control}
 \dot{x}(t) = F(x(t))+ K \left[x(t) - hx(t-\Theta(h)p) \right]
\end{equation} 
with $K \in \mathbb{R}^{N \times N}$. 
We additionally make the mild assumption that the matrix $K \in \mathbb{R}^{N \times N}$ satisfies $hK  = Kh$, so that the system \eqref{eq:equivariant control} is again equivariant with respect to the group generated by $h$. 

In system \eqref{eq:equivariant control}, the delay $\Theta(h)p$ is strictly smaller then the minimal period of $x_\ast$, and hence we are \emph{not} in the setting of Theorem \ref{thm:periodicity}. However, Theorem \ref{thm:main result} gives a characteristic matrix function $z \mapsto \Delta(z)$ for the operator
\[ h^{-1} U(\Theta(h)p, 0) \]
where $U(t, s), \ t \geq s$ is the family of solution operators of the DDE 
\begin{equation} \label{eq:control linearization}
\dot{y}(t) = F'(x_\ast(t)) y(t) + K \left[y(t) - hy(t-\Theta(h)p)\right].
\end{equation}
The eigenvalues of the operator $h^{-1} U(\Theta(h)p, 0)$ determine whether $x_\ast$ is stable as a solution of \eqref{eq:equivariant control} (cf. Proposition \ref{prop: cm delay equals period}); therefore, we can establish whether the control scheme \eqref{eq:equivariant control} succeeds or fails to stabilize $x_\ast$ by computing the roots of the equation $\det \Delta(z) = 0$ (see also Lemma \ref{lem:eigenvalues cm}). This result contributes to the current literature on equivariant Pyragas control in two ways: 
\begin{enumerate}
\item To prove that $x_\ast$ is an \emph{unstable} solution of \eqref{eq:equivariant control}, it suffices to find (at least) one eigenvalue of the operator $h^{-1} U(\Theta(h)p, 0)$ outside the unit circle; and in specific situations, it is indeed possible to do exactly that \cite{Hooton18}. 
However, if we want to establish that $x_\ast$ is a \emph{stable} solution of \eqref{eq:equivariant control}, we have to ensure that we find \emph{all} non-zero eigenvalues of $h^{-1} U(\Theta(h)p, 0)$ and have to be careful about the multiplicity of the trivial eigenvalue $1 \in \sigma_{pt}(h^{-1} U(\Theta(h)p, 0))$. Theorem \ref{thm:main result} paves a way to do that, since the characteristic matrix function captures all non-zero eigenvalues of $h^{-1} U(\Theta(h)p, 0)$ and also captures both their geometric and their algebraic multiplicity (cf. Lemma \ref{lem:eigenvalues cm}). 
\item In the literature so far, most analytical results on succesfull equivariant Pyragas control are either close to a bifurcation point \cite{Hooton19, Hooton17, deWolff2017, Fiedler20, refuting} or consider rotating waves, i.e. periodic solutions that can be transformed to stationary states of autonomous systems \cite{Purewal14,Fiedler08,equivariant,Fiedler10,Fiedlerconstraint}. Both these approaches simplify the stability analysis, but also work only in specific settings, i.e. they strongly depend on the form of the ODE \eqref{eq:ode pyragas}. In the context of equivariant Pyragas control, Theorem \ref{thm:main result} also simplifies the stability analysis by reducing the infinite dimensional problem to a finite dimensional one. However, this simplification is general in the sense that it does \emph{not} depend on the specific form of the ODE \eqref{eq:ode pyragas}. Therefore, we believe that Theorem \ref{thm:main result} is a first step in proving new stabilization results (such as the stabilization results for non-stationary periodic solutions and far away from bifurcation point in \cite{proefschrift}) and will generally be a helpful tool in further developments in equivariant Pyragas control. 
\end{enumerate}

\section{Discussion}

In \cite{Szalai06}, Szalai, St{\'e}p{\'a}n and Hogan discuss a delay equation of the form 
\[ \dot{x}(t) = f(x(t), x(t-\tau)), \qquad f:\mathbb{R}^2 \times \mathbb{R}^2 \to \mathbb{R}^2 \]
that has a periodic solution of period $2\tau$. To find geometrically simple eigenvalues of this periodic orbit, they construct a characteristic matrix function that takes values in $\mathbb{C}^{4 \times 4}$. In general, if the delay equation
\[ \dot{x}(t) = f(x(t), x(t-\tau)), \qquad f:\mathbb{R}^N \times \mathbb{R}^N \to \mathbb{R}^N \]
has a periodic orbit with period $\tau/m$, one expects that monodromy operator has a characteristic matrix function taking values in $\mathbb{C}^{(N \times m) \times (N \times m)}$, see also \cite{Sieber11}. In Theorem \ref{thm:main result}, in contrast, the period of the periodic orbit of \eqref{eq:dde thm 2} is rationally related to the delay, but the constructed characteristic matrix function takes values in $\mathbb{C}^{N \times N}$. The difference here is that we do not construct a characteristic matrix function for the monodromy operator, but exploit the equivariance relations and construct a characteristic matrix function for the operator \eqref{eq:reduced mon op}. So working with the operator \eqref{eq:reduced mon op} also has a computational advantage, since it yields a lower dimensional characteristic matrix function.

\medskip

Throughout this article, we studied stability of periodic orbits of DDE using the principle of linearized stability, i.e. by studying the behaviour of the linearized system. 
The advantage of this is that for linear DDE of the form 
\[ \dot{x}(t) = A(t) x(t) + B(t) x(t-\tau) \]
one can very explicitly compute the time $\tau$-map, cf. Lemma \ref{lem:expression T}. In contrast, a Poincar\'e map for periodic orbits of DDE can be constructed abstractly \cite[Section 14.3]{delayequations}, but in general no explicit expression for the Poincar\'e map is available.

\bibliographystyle{alpha}
\bibliography{biebcm}

\newcommand{\etalchar}[1]{$^{#1}$}
\begin{thebibliography}{DvGVW95}

\bibitem[DNES{\etalchar{+}}19]{Schoell19}
L.~Droenner, N.~Naumann, Eckehard E.~Sch{\"o}ll, A.~Knorr, and A.~Carmele.
\newblock Quantum {P}yragas control: Selective control of individual photon
  probabilities.
\newblock {\em Phys. Rev. A}, 99:023840, Feb 2019.

\bibitem[DvGVW95]{delayequations}
O.~Diekmann, S.~van Gils, S.~{Verduyn Lunel}, and H.~Walther.
\newblock {\em Delay Equations: Functional-, Complex-, and Nonlinear Analysis}.
\newblock Springer Verlag, 1995.

\bibitem[dW21]{proefschrift}
B.~de~{Wolff}.
\newblock {\em Delayed feedback stabilization with and without symmetry}.
\newblock PhD thesis, Freie Universit{\"a}t Berlin, 2021.

\bibitem[dWV17]{deWolff2017}
B.~de~{Wolff} and S.~{Verduyn Lunel}.
\newblock Control by time delayed feedback near a {H}opf bifurcation point.
\newblock {\em Electron. J. Qual. Theory Differ. Equ.}, (91):1--23, 2017.

\bibitem[FFG{\etalchar{+}}07]{refuting}
B.~Fiedler, V.~Flunkert, M.~Georgi, P.~H{\"o}vel, and E.~Sch{\"o}ll.
\newblock Refuting the odd-number limitation of time-delayed feedback control.
\newblock {\em Physical Review Letters}, 98, 2007.

\bibitem[FFG{\etalchar{+}}08]{Fiedler08}
Bernold Fiedler, Valentin Flunkert, Marc Georgi, Philipp H{\"o}vel, and
  Eckehard Sch{\"o}ll.
\newblock {\em Beyond the odd number limitation of time-delayed feedback
  control}, pages 73--84.
\newblock John Wiley \& Sons, 2008.

\bibitem[FFS10]{Fiedler10}
B.~Fiedler, V.~Flunkert, and E.~Sch{\"o}ll.
\newblock Delay stabilization of periodic orbits in coupled oscillator systems.
\newblock {\em Phil. Trans. Rol. Soc. A}, 368, 2010.

\bibitem[Fie88]{Fiedler88}
B.~Fiedler.
\newblock {\em Global bifurcation of periodic solutions with symmetry}.
\newblock Springer, 1988.

\bibitem[{Fie}08]{Fiedlerconstraint}
B.~{Fiedler}.
\newblock Time-delayed feedback control: qualitative promise and quantitative
  constraint.
\newblock {A.L. Fradkov et al. (eds.), 6th EUROMECH Conference on Nonlinear
  Dynamics ENOC 2008, Saint Petersburg, Russia, 2008.}, 2008.

\bibitem[FLR{\etalchar{+}}20]{Fiedler20}
B.~Fiedler, A.~{López Nieto}, R.~Rand, S.~Sah, I.~Schneider, and
  B.~de~{Wolff}.
\newblock Coexistence of infinitely many large, stable, rapidly oscillating
  periodic solutions in time-delayed duffing oscillators.
\newblock {\em Journal of Differential Equations}, 268(10):5969--5995, 2020.

\bibitem[HBKR17]{Hooton17}
E.~Hooton, Z.~Balanov, W.~Krawcewicz, and D.~Rachinskii.
\newblock Noninvasive stabilization of periodic orbits in o4-symmetrically
  coupled systems near a {H}opf bifurcation point.
\newblock {\em International Journal of Bifurcation and Chaos}, 27(06):1750087,
  2017.

\bibitem[HGTS21]{Schneider21}
D.~Herring, L.~Greten, J.~Totz, and I.~Schneider.
\newblock Equivariant {P}yragas control on networks of relaxation oscillators.
\newblock {\em to appear}, 2021.

\bibitem[HKR18]{Hooton18}
E.~Hooton, P.~Kravetc, and D.~Rachinskii.
\newblock Restrictions to the use of time-delayed feedback control in symmetric
  settings.
\newblock {\em Discrete {\&} Continuous Dynamical Systems - B}, 23(2):543--556,
  2018.

\bibitem[HKRH19]{Hooton19}
E.~Hooton, P.~Kravetc, D.~Rachinskii, and Q.~Hu.
\newblock Selective {P}yragas control of {H}amiltonian systems.
\newblock {\em Discrete and Continuous Dynamical Systems - S},
  12(7):2019--2034, 2019.

\bibitem[KV92]{KaashoekVL92}
M.~Kaashoek and S.~{Verduyn Lunel}.
\newblock Characteristic matrices and spectral properties of evolutionary
  systems.
\newblock {\em Transactions of the Americam Mathematical Society}, 334, 1992.

\bibitem[KV21]{KaashoekVL20}
M.~Kaashoek and S.~{Verduyn Lunel}.
\newblock {\em Completeness theorems, characteristic matrices and applications
  to integral and differential operators}.
\newblock Birkh{\"a}user, 2021.
\newblock to appear.

\bibitem[LI99]{Lamb99}
J.~Lamb and I.Melbourne.
\newblock Bifurcation from discrete rotating waves.
\newblock {\em Archive for Rational Mechanics and Analysis}, 149:229--270,
  1999.

\bibitem[PPK14]{Purewal14}
A.~Purewal, C.~Postlethwaite, and B.~Krauskopf.
\newblock A global bifurcation analysis of the subcritical {H}opf normal form
  subject to {P}yragas time-delayed feedback control.
\newblock {\em SIAM Journal on Applied Dynamical Systems}, 13(4):1879--1915,
  2014.

\bibitem[Pyr92]{Pyragas92}
K.~Pyragas.
\newblock Continuous control of chaos by self-controlling feedback.
\newblock {\em Physics Letters A}, 170(6):421--428, 1992.

\bibitem[San02]{Sandstede}
B.~Sandstede.
\newblock Stability of travelling waves.
\newblock In Bernold Fiedler, editor, {\em Handbook of Dynamical Systems},
  volume~2 of {\em Handbook of Dynamical Systems}, pages 983--1055. Elsevier
  Science, 2002.

\bibitem[SB16]{equivariant}
I.~Schneider and M.~Bosewitz.
\newblock Eliminating restrictions of time-delayed feedback control using
  equivariance.
\newblock {\em Discrete and Continuous Dynamical Systems}, 36, 2016.

\bibitem[SGH06]{Szalai06}
R.~Szalai, S.~Gábor, and S.~John Hogan.
\newblock Continuation of bifurcations in periodic delay‐differential
  equations using characteristic matrices.
\newblock {\em SIAM Journal on Scientific Computing}, 28(4):1301--1317, 2006.

\bibitem[SS11]{Sieber11}
J.~Sieber and R.~Szalai.
\newblock Characteristic matrices for linear periodic delay differential
  equations.
\newblock {\em SIAM Journal on Applied Dynamical Systems}, 10(1):129--147,
  2011.

\bibitem[SWLY13]{Sieber13}
J.~Sieber, M.~Wolfrum, M.~Lichtner, and S.~Yanchuk.
\newblock On the stability of periodic orbits in delay equations with large
  delay.
\newblock {\em Discrete {\&} Continuous Dynamical Systems}, 33(7):3109--3134,
  2013.

\bibitem[{Ver}92]{VL92}
S.~{Verduyn Lunel}.
\newblock Small solutions and completeness for linear functional differential
  equations.
\newblock {\em Oscillations and Dynamics in Delay Equations, vol. Contemporary
  Mathematics}, 129:127--152, 1992.

\bibitem[WS06]{Wulff2006}
C.~Wulff and A.~Schebesch.
\newblock Numerical continuation of symmetric periodic orbits.
\newblock {\em SIAM J. Appl. Dyn. Syst.}, 2006.

\end{thebibliography}

\end{document}